\documentclass[a4paper]{amsart}
\usepackage{textcomp, amsmath, amssymb, amsthm}
\usepackage{graphicx}
\usepackage{enumerate}
\usepackage[margin=0.95in]{geometry}

\usepackage{physics}
\usepackage[colorlinks,
linkcolor=black,
anchorcolor=black,
citecolor=black,
urlcolor=black
]{hyperref}

\usepackage{tikz} 

\tikzset{every picture/.style={line width=0.75pt}} 

\newtheorem{proposition*}{Proposition}

\theoremstyle{definition} 

\newtheorem{defin*}[proposition*]{Definition}
\newtheorem{lemma*}[proposition*]{Lemma}
\newtheorem{theorem*}[proposition*]{Theorem}
\newtheorem{corollary*}[proposition*]{Corollary}

\title{On the positivity of Fourier transform of the stretched Gaußian function}
\author{HANWEN LIU}
\date{}

\begin{document}

\maketitle

\begin{abstract}
    The stretched Gaußian function $f(\mathbf{x})=\exp \left(-\|\mathbf{x}\|^s\right)$, as a real function defined on $\mathbb{R}^d$, has found numerous applications in mathematics and physics. For instance, to describe results from spectroscopy or inelastic scattering, the Fourier transform of the stretched Gaußian function is needed. For $s \in(0,2]$, we prove that the Fourier transform of $f(\mathbf{x})=\exp \left(-\|\mathbf{x}\|^s\right)$ is everywhere positive on $\mathbb{R}^d$.
\end{abstract}

Throughout this article, let $\mathbb{N}$ be the set of non-negative integers, ant let $\left(\mathbb{R}^d,\langle\cdot, \cdot\rangle\right)$ be the $d$-dimensional standard Euclidean space with 2-norm $\|\cdot\|$. For any continuous real function $f$ defined on a subset $A$ of $\mathbb{R}$ and $n$ times differentiable at an interior point $x$ of $A$, denote by $f^{(n)}(x)$ the $n$-th derivative of $f$ at $x$.

Recall that the complete Bell polynomials are recurrently defined as
$$
B_{n+1}\left(x_1, \ldots, x_{n+1}\right):=\sum_{k=0}^n\begin{pmatrix}
n \\
k
\end{pmatrix} B_{n-k}\left(x_1, \ldots, x_{n-k}\right) x_{k+1}
$$
for all $n \in \mathbb{N}$, with the initial value $B_0=1$.

Using the complete Bell polynomials, a special case of Faà di Bruno's formula takes the following form:
\begin{proposition*}\label{proposition_1}
    Let $\Omega$ be a non-empty open subset of $\mathbb{R}$, and let $g\colon \Omega \rightarrow \mathbb{R}$ be a smooth function. For the smooth function
    $$
    \begin{aligned}
    f\colon \Omega & \longrightarrow \mathbb{R} \\
    x & \longmapsto \exp (g(x))
    \end{aligned}
    $$
    the formula
    $$
    f^{(n)}(x)=f(x) \cdot B_n\left(g^{(1)}(x), \ldots, g^{(n)}(x)\right)
    $$
    holds for all $n \in \mathbb{N}$ and $x \in \Omega$.
\end{proposition*}
\begin{proof}
    See for example \cite{1}.
\end{proof}

In order to prove our main theorem, we shall recall the definition of complete monotone functions.
\begin{defin*}\label{definition_2}
    A continuous function $f\colon[0,+\infty) \rightarrow \mathbb{R}$ is said to be complete monotone, if the restriction of $f$ on $(0,+\infty)$ is smooth, and $(-1)^n f^{(n)}(x) \geqslant 0$ for all $n \in \mathbb{N}$ and $x>0$.
\end{defin*}

An important example of complete monotone function is given in the lemma below.
\begin{lemma*}\label{lemma_3}
    Let $s \in(0,1]$. Then the continuous function
    $$
    \begin{aligned}
    f\colon[0,+\infty) & \longrightarrow \mathbb{R} \\
    x & \longmapsto \exp \left(-x^s\right)
    \end{aligned}
    $$
    is complete monotone.
\end{lemma*}
\begin{proof}
    Define $g\colon(0,+\infty) \rightarrow \mathbb{R}$ by $g(x):=-x^s$. Then we obtain
    $$
    (-1)^n g^{(n+1)}(x)=\dfrac{(-1)^{n+1}}{x^{n-s}} \dfrac{\Gamma(s+1)}{\Gamma(s-n+1)} \geqslant 0
    $$
    for all $n \in \mathbb{N}$ and $x>0$.

    By induction on $n \in \mathbb{N}$, we have that
    $$
    \begin{aligned}
    (-1)^n B_{n+1}\left(g^{(1)}(x), \ldots, g^{(n+1)}(x)\right) & =(-1)^n \sum_{k=0}^n\begin{pmatrix}
    n \\
    k
    \end{pmatrix} B_{n-k}\left(g^{(1)}(x), \ldots, g^{(n-k)}(x)\right) g^{(k+1)}(x) \\
    & =\sum_{k=0}^n\begin{pmatrix}
    n \\
    k
    \end{pmatrix}\left[(-1)^k g^{(k+1)}(x)\right](-1)^{n-k} B_{n-k}\left(g^{(1)}(x), \ldots, g^{(n-k)}(x)\right) \\
    & \geqslant 0
    \end{aligned}
    $$
    for all $n \in \mathbb{N}$ and $x>0$.

    Since $e^{g(x)}>0$ for all $x>0$, by Proposition \ref{proposition_1} we obtain
    $$
    (-1) f^{(n)}(x)=(-1)^n e^{g(x)} B_n\left(g^{(1)}(x), \ldots, g^{(n)}(x)\right) \geqslant 0
    $$
    for all $n \in \mathbb{N}$ and $x>0$.
\end{proof}

The following remarkable theorem of Bernstein characterizes the class of complete monotone functions.
\begin{theorem*}[Bernstein]\label{theorem_4}
    Let $f\colon[0,+\infty) \rightarrow \mathbb{R}$ be a continuous function. Then the following statements are equivalent:
    \begin{enumerate}
        \item[(i).] $f$ is complete monotone; 
        \item[(ii).] there exists a bounded increasing function $g\colon[0,+\infty) \rightarrow \mathbb{R}$ such that
        $$
        \int_0^{\infty} e^{-x t} \mathrm{d} g(t)=f(x)
        $$
        for all $x \in[0,+\infty)$.
    \end{enumerate}
\end{theorem*}
\begin{proof}
    See for example \cite{2}.
\end{proof}

Now we are ready to state and proof our main theorem.
\begin{corollary*}\label{corollary_5}
    Let $s \in(0,2]$. Then the Fourier transform $\hat{f}$ of the Lebesgue integrable function
    $$
    \begin{aligned}
    f\colon \mathbb{R}^d & \longrightarrow \mathbb{R} \\
    \mathbf{x} & \longmapsto \exp \left(-\|\mathbf{x}\|^s\right)
    \end{aligned}
    $$
    is everywhere positive on $\mathbb{R}^d$, that is,
    $$
    \hat{f}(\boldsymbol{\xi})=\int_{\mathbb{R}^d} f(\mathbf{x}) e^{2 \pi i\langle\mathbf{x}, \boldsymbol{\xi}\rangle} \mathrm{d} \mathbf{x}>0
    $$
    for all $\boldsymbol{\xi} \in \mathbb{R}^d$.
\end{corollary*}
\begin{proof}
    We shall first prove that $f \in L^1\left(\mathbb{R}^d\right)$. For any $r>0$, denote by $\partial B_r:=\left\{\mathbf{x} \in \mathbb{R}^d\colon\|\mathbf{x}\|=r\right\}$ the $(d-1)$-sphere of radius $r$ centered at 0, and denote by $\mathrm{d}\sigma$ its area element. Recall that the volume of the unit $(d-1)$-sphere $\partial B_1$ is 
    $$\int_{\partial B_1} \mathrm{d} \sigma = \dfrac{2 \pi^{d / 2}}{\Gamma(d / 2)}$$
    By the co-area formula,
    $$
    \begin{aligned}
    \int_{\mathbb{R}^d} f(\mathbf{x}) \mathrm{d} \mathbf{x} & =\int_{\mathbb{R}^d} \exp \left(-\|\mathbf{x}\|^s\right) \mathrm{d} \mathbf{x} \\
    & =\int_0^{\infty}\left(\int_{\partial B_x} \exp \left(-x^s\right) \mathrm{d} \sigma\right) \mathrm{d} x \\ 
    & =\dfrac{2 \pi^{d / 2}}{\Gamma(d / 2)} \int_0^{\infty} x^{d-1} \exp \left(-x^s\right) \mathrm{d} x \\ 
    & =\dfrac{2 \pi^{d / 2}}{s} \cdot \dfrac{\Gamma(d / s)}{\Gamma(d / 2)}
    \end{aligned}
    $$
    i.e. $f$ is Lebesgue integrable. In particular, $\hat{f}\colon \mathbb{R}^d \rightarrow \mathbb{R}$ is well-defined.

    Since $s / 2 \in(0,1]$, by Lemma \ref{lemma_3} the continuous function
    $$
    \begin{aligned}
    \varphi\colon [0,+\infty) & \longrightarrow \mathbb{R} \\
    x & \longmapsto \exp \left(-\sqrt{x}^s\right)
    \end{aligned}
    $$
    is completely monotone. Therefore by Theorem \ref{theorem_4} there exists a bounded increasing function $g\colon [0,+\infty) \longrightarrow \mathbb{R}$ such that
    $$
    \varphi(x)=\int_0^{\infty} e^{-x t} \mathrm{d} g(t)
    $$
    for all $x \in[0,+\infty)$. Without loss of generality, we assume $g(0)=0$. Notice that
    $$
    \begin{aligned}
    f(\mathbf{x}) & =\exp \left(-\|\mathbf{x}\|^s\right) \\
    & =\varphi\left(\|\mathbf{x}\|^2\right) \\
    & =\int_0^{\infty} \exp \left(-t\|\mathbf{x}\|^2\right) \mathrm{d} g(t)
    \end{aligned}
    $$
    for all $\mathbf{x} \in \mathbb{R}^d$.

    Fix any $\boldsymbol{\xi} \in \mathbb{R}^d$. Since the double integral
    $$
    \begin{aligned}
    \int_{\mathbb{R}^d} \int_0^{\infty}\left|\exp \left(-t\|\mathbf{x}\|^2-2 \pi i\langle\mathbf{x}, \boldsymbol{\xi}\rangle\right)\right| \mathrm{d} g(t) \mathrm{d} \mathbf{x} & =\int_{\mathbb{R}^d}\left(\int_0^{\infty} \exp \left(-t\|\mathbf{x}\|^2\right) \mathrm{d} g(t)\right) \mathrm{d} \mathbf{x} \\
    & =\int_{\mathbb{R}^d} f(\mathbf{x}) \mathrm{d} \mathbf{x} \\
    & <+\infty
    \end{aligned}
    $$
    by Fubini's theorem
    $$
    \begin{aligned}
    \int_{\mathbb{R}^d} f(\mathbf{x}) e^{-2 \pi i\langle \mathbf{x}, \boldsymbol{\xi}\rangle} \mathrm{d} \mathbf{x} & =\int_{\mathbb{R}^d}\left(\int_0^{\infty} \exp \left(-t\|\mathbf{x}\|^2\right) \mathrm{d} g(t)\right) e^{-2 \pi i\langle \mathbf{x}, \boldsymbol{\xi}\rangle} \mathrm{d} \boldsymbol{x} \\
    & =\int_0^{\infty}\left(\int_{\mathbb{R}^d} \exp \left(-t\|\mathbf{x}\|^2-2 \pi i\langle \mathbf{x}, \boldsymbol{\xi}\rangle\right) \mathrm{d} \mathbf{x}\right) \mathrm{d} g(t)
    \end{aligned}
    $$
    Moreover, since
    $$
    \int_{\mathbb{R}^d} \exp \left(-t\|\mathbf{x}\|^2-2 \pi i\langle\mathbf{x}, \boldsymbol{\xi}\rangle\right) \mathrm{d} \mathbf{x}=\left(\dfrac{\pi}{t}\right)^{d / 2} \exp \left(-\dfrac{\pi^2}{t}\|\boldsymbol{\xi}\|^2\right) \geqslant 0
    $$
    for all $t>0$, we obtain that
    $$
    \begin{aligned}
    \hat{f}(\boldsymbol{\xi}) & =\int_{\mathbb{R}^d} f(\mathbf{x}) e^{-2 \pi i\langle \mathbf{x}, \boldsymbol{\xi}\rangle} \mathrm{d} \mathbf{x} \\
    & =\int_0^{\infty}\left(\frac{\pi}{t}\right)^{d / 2} \exp \left(-\dfrac{\pi^2}{t}\|\boldsymbol{\xi}\|^2\right) \mathrm{d} g(t) \\
    & \geqslant 0.
    \end{aligned}
    $$
    It remains to prove that $\hat{f}(\boldsymbol{\xi}) \neq 0$. 
    
    Assume by contrary that $\hat{f}(\boldsymbol{\xi})$ vanishes. Take any $a, b \in \mathbb{R}$ such that $0<a<b$. By the mean value theorem for Riemann-Stieltjes integral, we have that
    $$
    \begin{aligned}
    0 & \leqslant[g(b)-g(a)] \inf _{a \leqslant t \leqslant b}\left(\frac{\pi}{t}\right)^{d / 2} \exp \left(-\frac{\pi^2}{t}\|\boldsymbol{\xi}\|^2\right) \\
    & \leqslant \int_a^b\left(\frac{\pi}{t}\right)^{d / 2} \exp \left(-\frac{\pi^2}{t}\|\boldsymbol{\xi}\|^2\right) \mathrm{d} g(t) \\ 
    & \leqslant \int_0^{\infty}\left(\frac{\pi}{t}\right)^{d / 2} \exp \left(-\frac{\pi^2}{t}\|\boldsymbol{\xi}\|^2\right) \mathrm{d} g(t) \\ 
    & =\hat{f}(\boldsymbol{\xi}) \\ 
    & =0
    \end{aligned}
    $$
    i.e. $[g(b)-g(a)] \cdot \inf \left\{\left(\frac{\pi}{t}\right)^{d / 2} \exp \left(-\frac{\pi^2}{t}\|\boldsymbol{\xi}\|^2\right)\colon a \leqslant t \leqslant b\right\}=0$. Since
    $$ \inf _{a \leqslant t \leqslant b}\left(\frac{\pi}{t}\right)^{d / 2} \exp \left(-\frac{\pi^2}{t}\|\boldsymbol{\xi}\|^2\right)=\left(\dfrac{\pi}{a}\right)^{d / 2} \exp \left(-\dfrac{\pi^2}{a}\|\boldsymbol{\xi}\|^2\right)>0$$
    we obtain $g(a)=g(b)$.

    Therefore there exists $\lambda \in \mathbb{R}$ such that $g(0)=0$ and $g(x) \equiv \lambda$ for all $x>0$. This proves that
    $$
    \begin{aligned}
    \exp \left(-\sqrt{x}^s\right) & =\varphi(x) \\
    & =\int_0^{\infty} e^{-x t} \mathrm{d} g(t) \\
    & =\int_0^{\infty} \delta(x) e^{-x t} \mathrm{d} t \\
    & =\lambda \cdot \exp (-x \cdot 0) \\
    & \equiv \lambda
    \end{aligned}
    $$
    a contradiction.

    Therefore we conclude that $\hat{f}(\boldsymbol{\xi})>0$ for all $\boldsymbol{\xi} \in \mathbb{R}^d$.
\end{proof}

\vspace{25pt}

University of Warwick, Coventry, CV4 7AL, UK

\href{hanwen.liu@warwick.ac.uk}{hanwen.liu@warwick.ac.uk}

\end{document}